\documentclass[11pt,reqno]{amsart}
\usepackage{graphicx}
\usepackage[pdftex]{hyperref}
\usepackage{subfigure}
\usepackage{amsthm}
\usepackage[margin=2cm]{geometry}
\usepackage{palatino}

\newtheorem{theorem}{Theorem}[section]
\newtheorem{lemma}[theorem]{Lemma}

\theoremstyle{definition}

\theoremstyle{remark}
\newtheorem{remark}[theorem]{Remark}
\newtheorem{notation}[theorem]{Notation}

\numberwithin{equation}{section}

\begin{document}

\title{A new Bernstein-type operator based on P\'olya's urn model with negative replacement}

\author{Mihai N. Pascu}
\address{Faculty of Mathematics and Computer Science, Transilvania University of Bra\c{s}ov, Str. Iuliu Maniu 50, Bra\c{s}ov -- 500091, Romania.}
\email{mihai.pascu@unitbv.ro}
\thanks{Supported by a grant of the Romanian National Authority for Scientific Research, CNCS - UEFISCDI, project number PNII-ID-PCCE-2011-2-0015.}

\author{Nicolae R. Pascu}
\address{Department of Mathematics, Kennesaw State University, 1100 S. Marietta Parkway, Marietta, GA 30060-2896, U.S.A.}
\email{npascu@kennesaw.edu}

\author{Floren\c{t}a Trip\c{s}a}
\email{florentatripsa@yahoo.com}
\address{Faculty of Mathematics and Computer Science, Transilvania University of Bra\c{s}ov, Str. Iuliu Maniu 50, Bra\c{s}ov -- 500091, Romania.}

\begin{abstract}
Using P\'{o}lya's urn model with negative replacement we introduce a new Bernstein-type operator and we show that the new operator improves upon the known estimates for the classical Bernstein operator. We also provide numerical evidence showing that the new operator gives a better approximation when  compared to some other classical Bernstein-type operators.
\end{abstract}



\maketitle


\section{Introduction}

About a hundred years ago, in his beautiful and short paper (\cite{Berstein 1912}, 2 pages), Serge Bernstein gave a (probabilistic) proof of Weierstrass's theorem on uniform approximation by polynomials, with a constructive method of approximation, known nowadays as Bernstein's polynomials.


The probabilistic idea behind Bernstein's construction can be seen as
follows. If $X_{n}$ is random variable with a binomial distribution
with parameters $n\in \mathbb{N}^{\ast }$ (number of trials) and $p\in \left[
0,1\right] $ (probability of success), then $E\left( \frac{X_{n}}{n}\right) =p$. Choosing $p=x\in \left[ 0,1%
\right] $, we have $E\left(\frac{X_{n}}{n}\right) =x$, and since the variance $%
\sigma ^{2}\left( \frac{X_{n}}{n}\right) =\frac{x\left( 1-x\right)}{n}$ is small for $n$
sufficiently large, heuristically we have $\frac{X_{n}}{n}\approx x$, and
if $f:[0,1]\rightarrow \mathbb{R}$ is continuous we also
have $f\left( \frac{X_{n}}{n}\right) \approx f\left( x\right) $. Taking
expectation leads to Bernstein's polynomials%
\begin{equation}\label{Probabilistic repr of Bernstein polynomial}
B_{n}\left( f;x\right) =\sum_{k=0}^{n}f\left( \frac{k}{n}\right)
C_{n}^{k}x^{k}\left( 1-x\right) ^{n-k} =Ef\left( \frac{X_{n}}{n}\right) \approx f(x),
\end{equation}%
and Bernstein's proof shows that this intuition is indeed correct: if $f$ is
continuous on $\left[ 0,1\right] $, then $B_{n}$ converges uniformly to $f$
on $\left[ 0,1\right] $ for $n\rightarrow \infty $.

Aside from their importance in Analysis, Bernstein's polynomials generated
an important area of research in various fields of Mathematics and Computer
Science, which continues to develop even today. The Bernstein polynomials
were intensively studied in Operator Theory and Approximation Theory, where
they were generalized by several authors, for example by F. Schurer
(Bernstein-Schurer operator, \cite{Schurer}), D. D. Stancu (Bernstein-Stancu
operator, \cite{Stancu}), A. Lupa\c{s} (Lupa\c{s} operator, \cite{Lupas},
and q-Bernstein operator, \cite{Lupas-q-Bernstein}), G. M. Phillips ($q$%
-Bernstein operator, \cite{Phillips}), M. Mursaleen et. al. ($\left(
p,q\right) $-Bernstein operator, \cite{Mursaleen}), and many others. See also \cite{Farouki} for a recent survey of Bernstein polynomials from the historical prospective,
of the properties and algorithms of interest in Computer Science, and of their various applications.

In the present paper we are concerned with a generalization of Bernstein polynomials based on Polya's urn distribution with (negative) replacement, which to our knowledge is new in the literature, the primary interest being the study of a new operator obtained for a particular choice of parameters involved. Our main results (Theorem \ref{Theorem estimate using modulus of continuity of f}, Theorem \ref{Theorem estimate using modulus of continuity of derivative of f}, and Theorem \ref{Theorem of asymptotic behaviour}) indicate  that the new operator $R_n$ improves the approximation provided by the classical Bernstein operator, and the numerical results (Section \ref{Section Numerical results}) also show that the new operator gives a better approximation than some of the well-known Bernstein-type operators.

The structure of the paper is as follows. In Section \ref{Section Polya urn
model} we set up the notation and we review the basic properties of the
P\'{o}lya urn model, which will be used in the sequel.

In Section \ref{Section Polya's distribution and probabilistic operators} we
introduce the operator $P_{n}^{a,b,c}$ depending on $a,b\in \mathbb{R}_{+}$
and $c\in \mathbb{R}$ satisfying an additional hypothesis. Although for $%
c\geq 0$ the operator $P_{n}^{x,1-x,c}$ coincides with the classical
Bernstein-Stancu operator (\cite{Stancu}), our primary interest in the present paper
is to consider negative values of $c$, for which the resulting operator $%
R_{n}$ seems to have better approximation properties than other Bernstein-type
operators (see Remark \ref{remark on condition c>0}, and the results in
Section \ref{Section error estimates for the operator R_n} and Section \ref%
{Section Numerical results}).

The main properties of the operator $R_{n}$ are given in Section \ref%
{Section Some properties of the operator R_n}, Theorem \ref{Theorem on main
properties of R_n}.

In Section \ref{Section error estimates for the operator R_n} we give the
error estimates for the operator $R_{n}$. Using a probabilistic lemma which
may be of independent interest (Lemma \ref{Probabilistic lemma}), in Theorem %
\ref{Theorem estimate using modulus of continuity of f} we give a short
proof of an error estimate for $R_{n}$ using the modulus of continuity. The
constant involved ($C=\frac{31}{27}\approx 1.14815...$) is smaller than the
corresponding constant in the case of Bernstein polynomial $B_{n}$ obtained
by Popoviciu ($C=\frac{3}{2}$) and Lorentz ($C=\frac{5}{4}$), but it is sligtly
larger than the optimal constant $C_{opt}\approx 1.08988...$ obtained by
Sikkema. In a subsequent paper (\cite{PPT}), we will show that the actual value of the constant is in fact smaller than the optimal constant obtained by Sikkema in the case of the classical Bernstein operator. In Theorem \ref{Theorem estimate using modulus of continuity of derivative of f} we give the error estimate in the case of a continuously differentiable function, and in Theorem \ref{Theorem of asymptotic behaviour} we give the asymptotic behaviour of the error estimate in the case of a twice continuously differentiable function.


The paper concludes (Section \ref{Section Numerical results}) with several numerical results which also indicate that the operator $R_{n}$ provides a better approximation than other classical Bernstein-type operators, even for small values of $n$ or discontinuous functions.

\section{Preliminaries}\label{Section Polya urn model}

P\'{o}lya's urn model (also known as P\'{o}lya-Eggenberger urn model, see
\cite{Eggenberger-Polya}, \cite{Polya}) genera\-lizes the classical urn model,
in which one observes balls extracted from an urn containing balls of two
colors. Urn models have been considered by various authors, including
Bernstein (\cite{Bernstein-urn models}) - see for example Friedman's model (%
\cite{Friedman}, or \cite{Freedman} for an extension of it), or the recent
paper \cite{Janson} on generalized P\'{o}lya urn models, and the references
cited therein.

For sake of completeness, we briefly describe the P\'{o}lya's urn model
which will be used in the sequel. The simplest urn model is the case when
balls are extracted successively from an urn containing balls of two
colors ($a$ white balls and $b$ black balls, $a,b\in \mathbb{N}$), the extracted ball being returned to the urn
before the next extraction. In this case, the
number of white balls in $n\geq 1$ extractions from the urn follows a
binomial distribution with parameters $n$ and $p=\frac{a}{a+b}$ (considering the extraction of a white ball a
``success''). 

In P\'{o}lya's urn model, the extracted ball is returned to the urn together with $c$ balls of the same
color, the case of a negative integer $c\in \mathbb{Z}$ being interpreted as removing $\left\vert c\right\vert $ balls from the urn. When $c$ is negative, the model breaks down if there are insufficient many balls of the desired color in the urn, the conditions for which the model is meaningful (also indicated in original P\'{o}lya's paper) being
\begin{equation}\label{Hypothesis on c}
a+\left( n-1\right) c\geq 0\qquad \text{and }\qquad b+\left( n-1\right)
c\geq 0.
\end{equation}

The physical model described above assumes $a,b,c$ to be integers, but probabilistically the model makes sense (defines a distribution) for $a,b\in \mathbb{R}_{+}$
and $c\in \mathbb{R}$, with the additional hypothesis (\ref{Hypothesis on c}, which we will assume in the sequel. It is easy to see that the binomial distribution corresponds to the case $c=0$ in P\'olya's urn model.



\begin{notation}
Since some authors use the same notation with different meanings, we first set the notation used in the sequel. For $%
x,h\in \mathbb{R}$ and $n\in \mathbb{N}$ we set
\begin{equation}
x^{\left( n,h\right) }=x\left( x+h\right) \left( x+2h\right) \cdot \ldots
\cdot \left( x+\left( n-1\right) h\right)   \label{rising factorial}
\end{equation}%
for the generalized (rising) factorial with increment $h$. We
are using the convention that an empty product is equal to $1$, that is $%
x^{\left( 0,h\right) }=1$ for any $x,h\in \mathbb{R}$.

\end{notation}

A random variable $X_{n}^{a,b,c}$ with P\'{o}lya's urn distribution with parameters $ n\geq 1$, $a,b\in \mathbb{R}_{+}$, and $c\in\mathbb{R}$ satisfying (\ref{Hypothesis on c}) is given by (see for example \cite{Johnson-Kotz})
\begin{equation}
P\left(X_n^{a,b,c}=k\right)=p_{n,k}^{a,b,c}=C_{n}^{k}\frac{\left( a\right) ^{\left( k,c\right) }\left(
b\right) ^{\left( n-k,c\right) }}{\left( a+b\right) ^{\left( n,c\right) }}%
,\qquad k\in \left\{ 0,1,\ldots ,n\right\} .  \label{Polya urn probabilities}
\end{equation}




It is known (see for example \cite{Johnson-Kotz}) that the mean and variance of $X_n^{a,b,c}$  are given by
\begin{equation}
E\left( X_{n}^{a,b,c}\right) =\frac{na}{a+b}\qquad \text{and} \qquad \sigma
^{2}\left( X_{n}^{a,b,c}\right) =\frac{nab}{\left( a+b\right) ^{2}}\left( 1+\frac{%
\left( n-1\right) c}{a+b+c}\right) .  \label{Polya mean and variance}
\end{equation}

\section{A new Bernstein-type operator\label{Section Polya's distribution and probabilistic operators}}

For $a,b\in\mathbb{R}$ with $a<b$ denote by $\mathcal{F}\left( \left[ a,b\right] \right) $ the set of
real-valued functions defined on $\left[ a,b\right] $, and by $C\left( \left[a,b\right] \right)$ the set of real-valued continuous functions on $\left[ a,b \right] $.

Consider the operator $P_{n}^{a,b,c}:\mathcal{F}\left( \left[ 0,1\right]
\right) \rightarrow \mathcal{F}\left( \left[ 0,1\right] \right) $, defined by%
\begin{equation}
P_{n}^{a,b,c}\left( f;x\right) =Ef\left( \frac{1}{n}X_{n}^{a,b,c}\right)
=\sum_{k=0}^{n}p_{n,k}^{a,b,c}f\left( \frac{k}{n}\right) ,\qquad f\in
\mathcal{F}\left( \left[ 0,1\right] \right) ,\quad x\in \left[ 0,1\right] ,
\label{Polya-Bernstein operator}
\end{equation}%
where the parameters $a,b,c$ may depend on $n$ and $x$, and satisfy $a,b\geq
0$ and the condition (\ref{Hypothesis on c}). In view of the probabilistic
representation above, we may call the operator $P_{n}^{a,b,c}$ a
P\'{o}lya-Bernstein type operator. Note that if the parameters $a,b,c$ depend
continuously on $x\in \left[ 0,1\right] $, from (\ref{Polya urn
probabilities}) it follows the operator $P_{n}^{a,b,c}$ maps an arbitrary
function to a continuous function, and in particular it maps continuous
functions to continuous functions.

Note that in the case $a=x$, $b=1-x$ and $c=\alpha \geq 0$ the above is the
probabilistic representation of the Bernstein-Stancu operator (introduced in
\cite{Stancu})%
\begin{equation}
P_{n}^{\langle \alpha \rangle }\left( f;x\right) =P_{n}^{x,1-x,\alpha
}\left( f;x\right) =\sum_{k=0}^{n}C_{n}^{k}\frac{x^{\left( k,\alpha \right)
}\left( 1-x\right) ^{\left( n-k,\alpha \right) }}{1^{\left( n,\alpha \right)
}}f\left( \frac{k}{n}\right) ,  \label{Bernstein-Stancu operator}
\end{equation}%
which generalizes the classical Bernstein operator $B_{n}\left( f;x\right) $
(the case $\alpha =0$).

The choice $a=x$, $b=1-x$, and $c=1/n$ gives the probabilistic
representation of the Lupa\c{s} operator (introduced in \cite{Lupas})%
\begin{equation}
P_{n}^{\langle 1/n\rangle }\left( f;x\right) =P_{n}^{x,1-x,1/n}\left(
f;x\right) =\sum_{k=0}^{n}C_{n}^{k}\frac{x^{\left( k,1/n\right) }\left(
1-x\right) ^{\left( n-k,1/n\right) }}{1^{\left( n,1/n\right) }}f\left( \frac{%
k}{n}\right) .  \label{Lupas operator}
\end{equation}

Other generalizations of the Bernstein operator, initially based on the
so-called $q$-integers (and more recently by $\left( p,q\right) $-integers),
were first given by Lupa\c{s} (\cite{Lupas-q-Bernstein}), then by Phillips (%
\cite{Phillips}), and afterwards by several authors (see for example \cite%
{Agrawal}, \cite{Kim}, \cite{Muraru}, \cite{Mursaleen}, \cite{Phillips c}, and the references
cited therein).

\begin{remark}
\label{remark on condition c>0}As noted above, for $c\geq 0$ the operator $%
P_{n}^{x,1-x,c}$ is just the classical Bernstein-Stancu operator; however,
our main interest in the present paper is to consider the case $c<0$, which
does not seem to have been addressed in the literature. As the results in
Section \ref{Section error estimates for the operator R_n} show it, it is
precisely the case $c<0$ that improves the approximation results for
Bernstein-type operators. To see this, note that by Lemma \ref{Probabilistic
lemma}, the error of approximation for a P\'{o}lya-Bernstein operators of the
form (\ref{Polya-Bernstein operator}) is bounded by the variance of the
distribution $X_{n}^{a,b,c}$, which by (\ref{Polya mean and variance}) is an
increasing function of $c$. Although this is just an intuitive argument, our
result in Theorem \ref{Theorem estimate using modulus of continuity of f} (and the Remark \ref{remark on improvement of Bernstein approximation}) shows that for the choice $c=c\left( n,x\right) =-\frac{%
\min \left\{x,1-x\right\} }{n-1}$ which minimizes the variance $%
\sigma^2(X_n^{x,1-x,c})$ within the set of admissible values of $c$ given by
(\ref{Hypothesis on c}), the resulting operator gives better approximations
results than the classical Bernstein operator. Moreover,
the numerical results presented in Section \ref{Section Numerical results} suggest that this particular operator also provides better
approximation than other Bernstein-type operators mentioned above.

The case $c<0$ seems to have been overlooked in the literature, and there
are good reasons for it. The additional hypothesis which has to be imposed
in the case $c<0$ is (\ref{Hypothesis on c}), which, holding $a$ and $b$
fixed and letting $n\rightarrow \infty $ (justified by studies on the
asymptotic behavior of P\'{o}lya urn models, as studied by various authors,
for example \cite{Janson}), gives $c\geq 0$, the case considered by Stancu (%
\cite{Stancu}). To be precise, in \cite{Stancu} Stancu indicates that the
choice $\alpha =-\frac{1}{n}$ in (\ref{Bernstein-Stancu operator}) gives the
Lagrange interpolating polynomial, which cannot be used for the uniform
approximation of every continuous function on $\left[ 0,1\right] $, and
concludes with \textquotedblleft We will henceforth assume that the
parameter $\alpha $ is non-negative\textquotedblright .

Another reason is that with the choice $a=x$ and $b=1-x$, for an arbitrary $%
x\in \left[ 0,1\right] $ (considered by Stancu, Lupa\c{s}, and others), the
inequality (\ref{Hypothesis on c}) leads again to the condition $c\geq 0$.
\end{remark}

We consider the particular choice $a=x$, $b=1-x$ and $c=-\min \left\{
x,1-x\right\} /\left( n-1\right) $ of the operator $P_{n}^{a,b,c}$ defined
above (note that for this choice of parameters the inequality (\ref%
{Hypothesis on c}) is satisfied for all $n>1$ and $x\in \left[ 0,1\right] $%
), and denote by $R_{n}:\mathcal{F}\left( \left[ 0,1\right] \right)
\rightarrow C\left( \left[ 0,1\right] \right) $ the operator which maps $%
f\in \mathcal{F}\left( \left[ 0,1\right] \right) $ to
\begin{eqnarray}  \label{Rational Bernstein operator}
R_{n}\left( f;x\right) &=&Ef\left( \frac{1}{n}X_{n}^{x,1-x,-\min \left\{
x,1-x\right\} /(n-1)}\right) \\
&=&\sum_{k=0}^{n}C_{n}^{k}\frac{x^{\left( k,-\min
\left\{ x,1-x\right\} /(n-1)\right) }\left( 1-x\right) ^{\left( n-k,-\min
\left\{ x,1-x\right\} /(n-1)\right) }}{1^{\left( n,-\min \left\{
x,1-x\right\} /(n-1)\right) }}f\left( \frac{k}{n}\right).\notag
\end{eqnarray}

The only downsize in considering the non-positive value $c=-\min \left\{
x,1-x\right\} /\left( n-1\right)$ above is that the operator $R_{n}$
is no longer a polynomial operator in $x$, but rather a rational operator:
on each of the intervals $\left[ 0,1/2\right] $ and $\left[ 1/2,1\right] $, $%
R_{n}\left( f;x\right) $ is a ratio of a polynomial of degree at most $n$ in
$x$ and the polynomial $1^{\left( n,-\min \left\{x,1-x\right\} /(n-1)\right)
}$ of degree $n-1$ in $x$, which does not depend on $f$. However, the
advantages of our choice are that it produces better approximation results
than other classical operators (see the various error estimates for the
operator $R_n$ given in Section \ref{Section error estimates for the
operator R_n} and the numerical results in Section \ref{Section Numerical results}), and from the point of view of applications the operator $R_n$
is as easily computable as a polynomial operator. For example, comparing (%
\ref{Rational Bernstein operator}) and (\ref{Bernstein-Stancu operator}) it
is easy to see that in order to evaluate $R_{n}\left( f;x\right) $ for a
fixed $n>1$ and $x\in \left[ 0,1\right] $, one can compute $c=-\min
\left\{ x,1-x\right\} /\left( n-1\right) $, and then evaluate $%
P_{n}^{\langle c\rangle }\left( f;x\right) $. The number of operations
needed for evaluating the operator $R_{n}\left( f;x\right) $ is thus one
unit more than the number of operations needed for evaluating the
Bernstein-Stancu polynomial operator. 

\section{Some properties of the operator $R_{n}$\label{Section Some
properties of the operator R_n}}

The first properties of the operator $R_{n}$ are given by the following.

\begin{theorem}
\label{Theorem on main properties of R_n}For any $n> 1$, $R_{n}:\mathcal{F%
}\left( \left[ 0,1\right] \right) \rightarrow C\left( \left[ 0,1\right]
\right) $ is a positive linear operator, which maps the test functions $%
e_{0}\left( x\right) \equiv 1$, $e_{1}\left( x\right) \equiv x$, and $%
e_{2}\left( x\right) \equiv x^{2}\ $respectively to
\begin{equation*}
R_{n}\left( e_{0};x\right) =1, \; R_{n}\left(
e_{1};x\right) =x,\;R_{n}\left( e_{2};x\right) =\frac{(n-1)x^2}{n}+\frac{x(\left( n-1\right)-n \min \left\{ x,1-x\right\} )}{n(n-1-\min \left\{ x,1-x\right\} )}.
\end{equation*}%

In particular, if $f:\left[ 0,1\right] \rightarrow \mathbb{R}$ is
continuous, then $R_{n}\left( f;x\right) $ converges to $f\left( x\right) $
uniformly on $\left[ 0,1\right] $ as $n\rightarrow \infty $.

Further, if $f:[0,1]\rightarrow \mathbb{R}$ is a convex function, $%
R_{n}\left( f;x\right) \geq f\left( x\right) $ for any $x\in \left[ 0,1%
\right] $.
\end{theorem}

\begin{proof}
The first claim follows easily from  the definition (\ref{Rational Bernstein operator}) of $%
R_{n}$, using the linearity and positivity of the
expected value. Using again the probabilistic representation of $R_{n}$ and the
properties (\ref{Polya mean and variance})
of P\'{o}lya's distribution we obtain:%
\begin{eqnarray*}
R_{n}\left( e_0;x\right) &=&E\left( 1\right) =1, \\
R_{n}\left( e_1;x\right) &=&\frac{1}{n}E\left( X_{n}^{x,1-x,-\min \left\{
x,1-x\right\} /\left( n-1\right) }\right) =x, \\
R_{n}\left( e_2;x\right) &=&\frac{1}{n^{2}}E\left( \left( X_{n}^{x,1-x,-\min \left\{ x,1-x\right\} /\left( n-1\right) }\right) ^{2}\right) 
=\frac{(n-1)x^2}{n}+\frac{x\left(n-1-n \min \left\{x,1-x\right\}\right) }{n\left(n-1-\min \left\{ x,1-x\right\}\right) }
\end{eqnarray*}

Since $R_n(e_i;x)\underset{n\rightarrow\infty}\longrightarrow e_i(x)$ uniformly on $[0,1]$ for $i=0,1,2$, the second part follows now from preceding part of the theorem using the classical Bohman-Korovkin theorem.

If $f$ is convex, by Jensen's inequality we obtain%
\begin{equation*}
R_{n}\left( f;x\right) \geq f\left(
E\left( \frac{1}{n}X_{n}^{x,1-x,-\min \left\{ x,1-x\right\} /\left(
n-1\right) }\right) \right) =f\left( x\right) ,
\end{equation*}
concluding the proof.
\end{proof}


\begin{remark}
The previous result can be generalized immediately. A similar proof shows that more
generally, a probabilistic operator $\mathcal{L}_{n}$ of the form%
\begin{equation}
\mathcal{L}_{n}\left( f;x\right) =E\left( f\left( X_n\right) \right) ,
\label{probabilistic operator}
\end{equation}%
where $f$ is a given continuous function defined on an closed interval containing the range of the random variable $X_{n,x}$ (whose distribution
depends on $x$ and $n$), with $E\left( X_{n,x}\right)=x$, is a positive linear operator, and satisfies%
\begin{equation}  \label{conditions on general operators}
\mathcal{L}_{n}\left( e_0;x\right) =1\qquad \text{and}\qquad \mathcal{L}%
_{n}\left( e_1;x\right) =x.
\end{equation}

If in addition to the above $\mathcal{L}_{n}\left( e_2;x\right) $ converges
uniformly to $x^{2}$ as $n\rightarrow \infty $, one easily deduces the
uniform convergence of $\mathcal{L}_{n}\left( f;x\right) $ to $f\left(
x\right) $. Further, if $f$ is convex and $f\left( X-{n,x}\right) $ has finite
mean, then by Jensen's inequality we also have $\mathcal{L}_{n}\left(
f;x\right) \geq f\left( x\right) $, for $x$ in the domain of $f$.

Under mild assumptions, the converse of this result also holds. More precisely, if $\mathcal{L}:C\left( \left[ a,b\right] \right) \rightarrow C\left( \left[ a,b\right] \right) $ is a bounded positive linear operator which satisfies (\ref{conditions on general operators}),
then there exists a random variable $X_x$ (whose distribution depends on a parameter $x\in[a,b]$) with values in $[a,b]$ and with mean $E\left( X_x\right) =x$ such that for any $f\in C\left( \left[ a,b\right] \right) $ we have
\[
\mathcal{L}_n\left( f;x\right) =E\left( f\left( X_x\right) \right)
,\qquad x\in \left[ a,b\right] .
\]
The proof follows from the Riesz representation theorem and the hypotheses in (\ref{conditions on general operators}).
\end{remark}

There are many approximation operators in the literature, and the above
remark shows that under the given hypotheses, one can attach a probabilistic
representation to them. In turn, this gives a better insight on their
properties, and it can simplify certain computation. For example,
probabilistic operators of the form (\ref{probabilistic operator}) can be
easily approximated by means of Monte-Carlo methods, which is important in
practical applications where these operators are being used.

\section{Error estimates for the approximation by the operator $R_{n}\left(
f;x\right)$ \label{Section error estimates for the operator R_n}}

T. Popoviciu (\cite{Popoviciu}) proved the following bound for the
approximation error for the Bernstein polynomial in the case of an arbitrary
continuous function $f:\left[ 0,1\right] \rightarrow \mathbb{R}$%
\begin{equation}  \label{Popoviciu's error estimate for Bernstein}
\left\vert B_{n}\left( f;x\right) -f\left( x\right) \right\vert \leq C\omega
\left( n^{-1/2}\right) ,\qquad x\in \left[ 0,1\right] , \; n=1,2,\ldots,
\end{equation}%
with $C=\frac{3}{2}$, where $\omega \left( \delta \right) =\omega ^{f}\left(
\delta \right) =\max \left\{ \left\vert f\left( x\right) -f\left( y\right)
\right\vert :x,y\in \left[ 0,1\right] \text{, }\left\vert x-y\right\vert \le
\delta \right\} $ denotes the modulus of continuity of $f$. Lorentz (\cite%
{Lorentz}, pp. 20 --21) improved the value of the constant $C$ above to $%
\frac{5}{4}$, and showed that the constant $C$ cannot be less than one. The
optimal value of the constant $C$ for which the inequality (\ref{Popoviciu's
error estimate for Bernstein}) holds true for any continuous function was
given by Sikkema (\cite{Sikkema}), who obtained the value
\begin{equation}
C_{opt}=\frac{4306+837\sqrt{6}}{5932}\approx 1.0898873...,
\label{Sikkema optimal constant}
\end{equation}%
attained in the case $n=6$ for a particular choice of $f$. We will show that the
operator $R_{n}$ defined by (\ref{Rational Bernstein operator}) also
satisfies a Popoviciu type inequality.

In order to give the result, we begin with the following auxiliary lemma which may be of independent interest. We note that although related estimates appear in the literature, we could not find a reference for them in the present form. For example, a result in the same spirit with a) below appears in \cite[Theorem 1]{Khan}, but there $\delta=n^{-1/2}$, and the right handside is replaced by the supremum of the corresponding inequality in (\ref{Khan's inequality}).

\begin{lemma}
\label{Probabilistic lemma}Let $X$ be a discrete random variable taking
values in an interval $\left[ a,b\right] \subset \mathbb{R}$, with finite
mean $E\left( X\right) =x$ and variance $\sigma ^{2}\left( X\right) $, and
let $f:\left[ a,b\right] \rightarrow \mathbb{R}$ for which $f\left( X\right)
$ has finite mean.

\begin{itemize}
\item[a)] If $f$ is continuous on $\left[ a,b\right] $, the for any $\delta
>0$ we have%
\begin{equation}\label{Khan's inequality}
\left\vert Ef\left( X\right) -f\left( x\right) \right\vert \leq \omega
\left( \delta \right) \left( 1+\frac{1}{\delta ^{2}}\sigma ^{2}\left(
X\right) \right) ,
\end{equation}
where $\omega \left( \delta \right) =\omega ^{f}\left( \delta \right) $
denotes the modulus of continuity of $f$.

\item[b)] If $f$ is continuously differentiable on $\left[ a,b\right] $, we
have%
\begin{equation}
\left\vert Ef\left( X\right) -f\left( x\right) \right\vert \leq \omega
_{1}\left( \delta \right) \left( \frac{1}{\delta }\sigma ^{2}\left( X\right)
+\sigma \left( X\right) \right) .
\end{equation}%
where $\omega _{1}\left( \delta \right) =\omega _{1}^{f}\left( \delta
\right) $ denotes the modulus of continuity of $f^{\prime }$.

\item[c)] Finally, if $f$ is twice continuously differentiable on $\left[ a,b%
\right] $, we have%
\begin{equation}
Ef\left( X\right) =f\left( x\right) +\frac{1}{2}f^{\prime \prime }\left(
x\right) \sigma ^{2}\left( X\right) +R\left( X\right) ,
\label{Asymptotic of Ef(X)}
\end{equation}%
and there exists $M>0$ such that for each $\varepsilon >0$ there exists $%
\delta =\delta \left( \varepsilon \right) >0$ such that%
\begin{equation}
\left\vert R\left( X\right) \right\vert \leq \varepsilon \sigma ^{2}\left(
X\right) +\left( b-a\right) ^{2}MP\left( \left\vert X-x\right\vert >\delta
\right) ,
\end{equation}%
where $M>0$ and $\delta =\delta \left( \varepsilon \right) >0$ depend on $f$%
, but not on $X$ or $x$.
\end{itemize}
\end{lemma}

\begin{proof}
Denoting by $F:\mathbb{R\rightarrow R}$ the distribution function of $X$, we
have%
\[
\left\vert Ef\left( X\right) -f\left( x\right) \right\vert
=\left\vert \int_{a}^{b}f\left( y\right) -f\left( x\right) dF\left( y\right)
\right\vert \leq \int_{a}^{b}\left\vert f\left( y\right) -f\left( x\right)
\right\vert dF\left( y\right) .
\]

It is not difficult to see that two arbitrary points $x,y\in \left[ a,b%
\right] $ are at most $\left[ \frac{\left\vert y-x\right\vert }{\delta }%
\right] +1$ intervals of length $\delta >0$ apart ($\left[ \frac{\left\vert
y-x\right\vert }{\delta }\right] \in \mathbb{N}$ denotes here the integer
part of $\frac{\left\vert y-x\right\vert }{\delta }$). Using this, the
definition of the modulus of continuity $\omega \left( \delta \right) $ of $%
f $, and the above, we obtain%
\begin{eqnarray*}
\left\vert Ef\left( X\right) -f\left( x\right) \right\vert &\leq &
\omega \left(\delta \right) \left( 1+\int_{a}^{b}\left[ \frac{\left\vert y-x\right\vert }{\delta }\right] dF\left( y\right) \right)
\leq \omega \left( \delta \right) \left( 1+\int_{a}^{b}\left[ \frac{\left\vert y-x\right\vert }{\delta }\right] ^{2}dF\left( y\right) \right)\\
&=&\omega \left( \delta \right) \left( 1+\int_{a}^{b}\frac{\left\vert y-x\right\vert }{\delta ^{2}}^{2}dF\left( y\right) \right) =\omega \left( \delta \right) \left( 1+\frac{1}{\delta^{2}}\sigma ^{2}\left( X\right) \right) ,
\end{eqnarray*}%
since by hypothesis $E\left( X\right) =x$.

To prove the second part of the lemma, by the mean value theorem we have%
\[
f\left( \alpha \right) -f\left( \beta \right) =f^{\prime }\left( \gamma
\right) \left( \alpha -\beta \right) =f^{\prime }\left( \beta \right) \left(
\alpha -\beta \right) +\left( f^{\prime }\left( \gamma \right) -f^{\prime
}\left( \beta \right) \right) \left( a-\beta \right) ,
\]%
for arbitrary points $\alpha ,\beta \in \left[ a,b\right] $, where $\gamma $ is an intermediate point between $\alpha $ and $\beta $.
Using this with $\alpha =y$ and $\beta =x$, we obtain%
\[
Ef\left( X\right) -f\left( x\right) =\int_{a}^{b}\left( f\left( y\right)
-f\left( x\right) \right) dF\left( y\right) =\int_{a}^{b}f^{\prime }\left(
x\right) \left( y-x\right) +\left( f^{\prime
}\left( \xi \right) -f^{\prime }\left( x\right) \right) \left( y-x\right)
dF\left( y\right) ,
\]%
where $\xi =\xi \left( x,y\right) $ is an intermediate point between $y$ and
$x$.

Applying a similar argument as above to the modulus of continuity $\omega
_{1}$ of $f^{\prime }$, and using the Cauchy-Schwarz inequality, we obtain%
\begin{eqnarray*}
&&\left\vert Ef\left( X\right) -f\left( x\right) \right\vert
\leq \left\vert f^{\prime }\left( x\right) \int_{a}^{b}y-xdF\left(
y\right) \right\vert +\omega _{1}\left( \delta \right) \int_{a}^{b}\left(
\left[ \frac{\left\vert \xi -x\right\vert }{\delta }\right] +1\right)
\left\vert y-x\right\vert dF\left( y\right) \\
&\leq &\left\vert f^{\prime }\left( x\right) \left( M\left( X\right)
-x\right) \right\vert +\omega _{1}\left( \delta \right) \left( \int_{a}^{b}%
\frac{\left\vert \xi -x\right\vert }{\delta }\left\vert y-x\right\vert
dF\left( y\right) +\int_{a}^{b}\left\vert y-x\right\vert dF\left( x\right)
\right) \\
&\leq &\omega _{1}\left( \delta \right) \left( \frac{1}{\delta }%
\int_{a}^{b}\left\vert y-x\right\vert ^{2}dF\left( y\right) +\left(
\int_{a}^{b}\left\vert y-x\right\vert ^{2}dF\left( x\right) \right)
^\frac12\right) =\omega _{1}\left( \delta \right) \left( \frac{1}{\delta }\sigma
^{2}\left( X\right) +\sigma \left( X\right) \right) .
\end{eqnarray*}

For the last part of the lemma, 
using Taylor's theorem we obtain%
\[
f\left( y\right) -f\left( x\right) =f^{\prime }\left( x\right) \left(
y-x\right) +\frac{1}{2}f^{\prime \prime }\left( x\right) \left( y-x\right)
^{2}+\alpha \left( y\right) \left( y-x\right) ^{2},\qquad y\in \left[ a,b%
\right] ,
\]%
where $\alpha :\left[ a,b\right] \rightarrow \mathbb{R}$ is bounded on $%
\left[ a,b\right] $, say by $M>0$, and satisfies $\lim_{y\rightarrow
x}\alpha \left( y\right) =0$. In particular, for any $\varepsilon >0$ there
exists $\delta =\delta \left( \varepsilon \right) >0$ such that $\left\vert
\alpha \left( y\right) \right\vert <\varepsilon $ for $\left\vert
y-x\right\vert <\delta $. Integrating the above in $y\in \left[ a,b\right] $, we obtain%
\begin{eqnarray*}
&&Ef\left( X\right) -f\left( x\right) =\int_{a}^{b}\left( f\left( y\right)
-f\left( x\right) \right) dF\left( y\right) \\
&=&f^{\prime }\left( x\right) \int_{a}^{b}\left( y-x\right) dF\left(
y\right) +\frac{1}{2}f^{\prime \prime }\left( x\right) \int_{a}^{b}\left(
y-x\right) ^{2}dF\left( y\right) +\int_{a}^{b}\alpha \left( y\right) \left(
y-x\right) ^{2}dF\left( y\right) \\
&=&f^{\prime }\left( x\right) \left( M\left( X\right) -x\right) +\frac{1}{2}%
f^{\prime \prime }\left( x\right) \sigma ^{2}\left( X\right)
+\int_{a}^{b}\alpha \left( y\right) \left( y-x\right) ^{2}dF\left( y\right)
\\
&=&\frac{1}{2}f^{\prime \prime }\left( x\right) \sigma ^{2}\left( X\right)
+\int_{a}^{b}\alpha \left( y\right) \left( y-x\right) ^{2}dF\left( y\right) .
\end{eqnarray*}

With $R\left( X\right) $ denoting the last integral above, we have%
\begin{eqnarray*}
\left\vert R\left( X\right) \right\vert &=&\left\vert \int_{a}^{b}\alpha
\left( y\right) \left( y-x\right) ^{2}dF\left( y\right) \right\vert \leq \varepsilon \sigma ^{2}\left( X\right) +M\int_{\substack{y\in \left[ a,b\right]
:\\ \left\vert y-x\right\vert \geq \delta \left( \varepsilon \right)} }\left( y-x\right) ^{2}dF\left( y\right) \\
&\leq &\varepsilon \sigma ^{2}\left( X\right) +M\left( b-a\right) ^{2}P\left( \left\vert X-x\right\vert >\delta \right) ,
\end{eqnarray*}%
concluding the proof.
\end{proof}

With this preparation we can now prove the first result, as follows.

\begin{theorem}
\label{Theorem estimate using modulus of continuity of f}If $f:\left[ 0,1%
\right] \rightarrow \mathbb{R}$ is a continuous function, then for any $n>1$
we have%
\begin{equation}
\left\vert R_{n}\left( f;x\right) -f\left( x\right) \right\vert \leq \omega
\left( n^{-1/2}\right) \left( 1+x\left( 1-x\right) \left( 1-\min \left\{
x,1-x\right\} \right) \right) ,\;\; x\in \left[ 0,1\right] ,
\label{bound involving modulus of continuity}
\end{equation}%
where $\omega \left( \delta \right) =\omega ^{f}\left( \delta \right) $
denotes the modulus of continuity of $f$.

In particular, we have%
\begin{equation}
\left\vert R_{n}\left( f;x\right) -f\left( x\right) \right\vert \leq \frac{31%
}{27}\omega \left( n^{-1/2}\right) ,\qquad x\in \left[ 0,1\right] .
\label{bound involving modulues of continuity biss}
\end{equation}
\end{theorem}

\begin{proof}
Applying part a) of Lemma \ref{Probabilistic lemma} with $\delta =n^{-1/2}$ and $X=%
\frac{1}{n}X_{n}^{x,1-x,\min \left\{ x,1-x\right\} /n-1}$, for which by (\ref%
{Polya mean and variance}) we have $EX=x$ and $\sigma ^{2}\left( X\right) =%
\frac{x\left( 1-x\right) }{n}\left( 1-\frac{\min \left\{ x,1-x\right\} }{%
1-\min \left\{ x,1-x\right\} /(n-1)}\right) $, we obtain%
\begin{eqnarray*}
\left\vert R_{n}\left( f;x\right) -f\left( x\right) \right\vert
&\le& \omega \left( n^{-1/2}\right) \left( 1+x\left( 1-x\right) \left( 1-\frac{%
\min \left\{ x,1-x\right\} }{1-\min \left\{ x,1-x\right\} /(n-1)}\right)
\right) \\
&\leq &\omega \left( n^{-1/2}\right) \left( 1+x\left( 1-x\right) \left(
1-\min \left\{ x,1-x\right\} \right) \right).
\end{eqnarray*}%

The expression $E\left( x\right) =x\left( 1-x\right) \left( 1-\min \left\{
x,1-x\right\} \right) $ satisfies $E\left( x\right) =E\left( 1-x\right) $,
for any $x\in \mathbb{R}$, thus $E\left( x\right) $ is a symmetric function
of $x$ with respect to $1/2$. For $x\in \left[ 0,1/2\right] $ we have $%
E\left( x\right) =x\left( 1-x\right) ^{2}$, with a maximum of $E\left(
1/3\right) =4/27$ at $x=1/3$. This shows that $E\left( x\right) \leq 4/27$
for $x\in \left[ 0,1\right] $, concluding the proof.
\end{proof}

\begin{remark}\label{remark on improvement of Bernstein approximation}
Note that the estimate (\ref{bound involving modulus of continuity}) improves the known estimate for the classical Bernstein operator (see for example \cite{Tonkov})
\begin{equation}
\left\vert B_{n}\left( f;x\right) -f\left( x\right) \right\vert \leq \omega
\left( n^{-1/2}\right) \left( 1+x\left( 1-x\right) \right) ,\;\; x\in \left[ 0,1\right] ,
\end{equation}
by the factor $\min \left\{x,1-x\right\}\le \frac12<1$.

Secondly, note that the value of the constant $C=\frac{31}{27}=1.14815$ in (\ref%
{bound involving modulues of continuity biss}) above is smaller than the
constants obtained by Popoviciu ($3/2$), respectively by Lorentz ($5/4$), in
the case of classical Bernstein polynomials, but it is slightly larger than the
optimal constant $C_{opt}\approx 1.0898873...$ obtained by Sikkema (\cite%
{Sikkema}). However, the bound in (\ref{bound involving modulues of
continuity biss}) is not optimal, and we chose to present it in this form
due to the simplicity of the proof. In a subsequent paper [\cite{PPT}) we will show
that the constant $C$ for which Popoviciu's type inequality (\ref{bound
involving modulues of continuity biss}) holds for any continuous function is actually
smaller than Sikkema's optimal constant for Bernstein polynomials. In turn,
this shows that the operator $R_{n}\left( f;x\right) $ improves the the well-known estimate for the classical Bernstein operator. Some related results
concerning the analogue of the estimate (\ref{bound involving modulus of
continuity}) in the case of Bernstein polynomial can be found in \cite%
{Tonkov}.
\end{remark}

The next result gives the estimation of the error for the operator $R_{n}$
in the case of a continuously differentiable function.

\begin{theorem}
\label{Theorem estimate using modulus of continuity of derivative of f}If $f:%
\left[ 0,1\right] \rightarrow \mathbb{R}$ is continuously differentiable on $%
\left[ 0,1\right] $, we have%
\begin{equation}
\left\vert R_{n}\left( f;x\right) -f\left( x\right) \right\vert \leq \frac{\omega
_{1}\left( n^{-1/2} \right)}{n^{1/2}} \left( x (1-x)(1-\min \left\{ x,1-x\right\}) + \sqrt{x(1-x) (1-\min \left\{ x,1-x\right\} )}\right),
\end{equation}%
for any $n>1$ and $x\in[0,1]$, where $\omega _{1}\left( \delta \right) $ denotes the modulus of continuity
of $f^{\prime }$.

In particular, we have%
\begin{equation}\label{estimate for differentiable functions}
\left\vert R_{n}\left( f;x\right) -f\left( x\right) \right\vert \leq \frac{%
4+6\sqrt{3}}{27}n^{-1/2}\omega _{1}\left( n^{-1/2}\right) ,\qquad x\in \left[ 0,1%
\right] .
\end{equation}
\end{theorem}

\begin{proof}
Applying part b) of Lemma \ref{Probabilistic lemma} with $X=\frac{1}{n}%
X_{n}^{x,1-x,-\min \left\{ x,1-x\right\} /(n-1)}$ and $\delta =n^{-1/2}$, and using (\ref{Polya mean and variance}), we
obtain%
\begin{eqnarray*}
&&\left\vert R_{n}\left( f;x\right) -f\left( x\right) \right\vert \leq
\omega _{1}\left( n^{-1/2}\right) \left( n^{1/2}\sigma ^{2}\left( X\right)
+\sigma \left( X\right) \right) \\
&\leq &n^{-1/2}\omega _{1}\left(n^{-1/2}\right) \left( x\left( 1-x\right)
\left( 1-\min \left\{ x,1-x\right\} \right) +\sqrt{x\left( 1-x\right) \left(
1-\min \left\{ x,1-x\right\} \right) }\right) .
\end{eqnarray*}

The same argument as in the last part of the proof of Theorem \ref{Theorem estimate using modulus of continuity of derivative of f} shows that the expression in parentheses
above has a maximum over $\left[ 0,1\right] $ equal to $\frac{4+6\sqrt{3}}{27}\approx 0.533$.
\end{proof}

\begin{remark}
The estimate corresponding to (\ref{estimate for differentiable functions}) in the case of Bernstein operator $B_n$ (see \cite{Lorentz}, p. 21) is given by
\[
\left\vert B_{n}\left( f;x\right) -f\left( x\right) \right\vert \leq \frac{3%
}{4}n^{-1/2}\omega _{1}\left( n^{-1/2}\right) ,\qquad x\in \left[ 0,1\right],
\]
and comparing to (\ref{estimate for differentiable functions}) we see that the operator $R_n$ improves upon the estimate for the Bernstein operator $B_n$ in the class of continuously differentiable functions on $[0,1]$.
\end{remark}

The following result gives the precise asymptotic of the error estimate for
the operator $R_{n}$ in the case of a twice continuously differentiable
function.

\begin{theorem}
\label{Theorem of asymptotic behaviour}If $f:\left[ 0,1\right] \rightarrow
\mathbb{R}$ is twice continuously differentiable on $\left[ 0,1\right] $,
then for any $n>1$ we have%
\begin{equation}
\lim_{n\rightarrow \infty }n\left( R_{n}\left( f;x\right) -f\left( x\right)
\right) =\frac{1}{2}f^{\prime \prime }\left( x\right) x\left( 1-x\right)\left( 1-\min \left\{ x,1-x\right\}\right) ,\qquad x\in \left[ 0,1\right].
\end{equation}
\end{theorem}

\begin{proof}
We will use the following recursion formula for the centered moments $\mu
_{k}=E\left( \left( X_n^{a,b,c}-E\left( X_{n}^{a,b,c}\right) \right) ^{k}\right) $ of P%
\'{o}lya's distribution $X_{n}^{a,b,c}$ with parameters $a,b,c$ and $n$ (see
e.g. \cite{Lorentz}, p. 191):%
\[
\mu _{k}=\sum_{j=0}^{k-2}C_{k-1}^{j}\left( \frac{c}{a+b}\mu _{j+1}+\delta
\mu _{j+1}+\gamma \mu _{j}\right) ,
\]%
where%
\begin{equation}
\gamma =np\left( 1-p\right) \left( 1+\frac{nc}{a+b}\right) ,\quad \delta
=n\left( 1-2p\right) \frac{c}{a+b}-p,\quad \text{and\quad }p=\frac{a}{a+b}%
\text{.}  \label{parameters for centered moments formula}
\end{equation}

Applying this with $a=x$, $b=1-x$, $c=-\min \left\{ x,1-x\right\} /(n-1)$
and $n$, for $k=3$ we obtain%
\begin{eqnarray*}
\mu _{3} &=&-\frac{\min \left\{ x,1-x\right\} }{n-1}\mu _{1}+\delta \mu
_{1}+\gamma \mu _{0}+2\left( \min \left\{ x,1-x\right\} \mu _{2}+\delta \mu
_{2}+\gamma \mu _{1}\right) \\
&=&\gamma +2\left( \min \left\{ x,1-x\right\} +\delta \right) \sigma
^{2}\left( X_{n}^{x,1-x,-\frac{\min\{x,1-x\}}{n-1}}\right) ,
\end{eqnarray*}%
which is of order $O\left( n\right) $ as $n\rightarrow \infty $, since by (%
\ref{Polya mean and variance}) we have $\sigma ^{2}\left( X_{n}^{x,1-x,-\frac{\min\{x,1-x\}}{n-1}}\right)
=O\left( n\right) $, and by (\ref%
{parameters for centered moments formula}) we have $\gamma 
=O\left( n\right) $,  $\delta 
=O\left( 1\right) $, and $p=x=O\left( 1\right) $.

With the same values for $a,b,c$ and taking $k=4$, the same formula gives%
\begin{equation*}
\mu _{4} 
=\gamma +3\sigma ^{2}\left( X_{n}^{x,1-x,-\frac{\min\{x,1-x\}}{n-1}}\right) \left( -\frac{\min \left\{
x,1-x\right\} }{n-1}+\delta +\gamma \right) +3\mu _{3}\left( -\frac{\min
\left\{ x,1-x\right\} }{n-1}+\delta \right) ,
\end{equation*}%
which shows that $\mu _{4}=O\left( n^{2}\right) $

Using part c) of Lemma \ref{Probabilistic lemma} with $X=\frac{1}{n}X_{n}^{x,1-x,-\frac{\min \left\{ x,1-x\right\}}{n-1} }$, $[a,b]=[0,1]$, and (\ref{Polya mean and variance}), we obtain%
\begin{equation*}
n\left(R_{n}\left( f;x\right) -f\left( x\right) \right)  
=\frac{1}{2}f^{\prime \prime }\left( x\right) x\left( 1-x\right) \left( 1-%
\frac{\min \left\{ x,1-x\right\} }{1-\frac{\min \left\{ x,1-x\right\}}{ n-1} }\right) +nR\left( \tfrac{1}{n}X_{n}^{x,1-x,-\frac{\min \left\{ x,1-x\right\}}{n-1} }\right) ,
\end{equation*}%
where for any $\varepsilon >0$ there exists $\delta =\delta \left(
\varepsilon \right) >0$ (which depends on $f$, but not on $n$ or $x$) such
that%
\begin{align*}
&\left\vert nR \left( \tfrac{1}{n}X_{n}^{x,1-x,-\frac{\min \left\{ x,1-x\right\}}{n-1} } \right) \right\vert \leq n\left(
\varepsilon \sigma ^{2}\left( \tfrac{1}{n}X_{n}^{x,1-x,-\frac{\min \left\{ x,1-x\right\}}{n-1} } \right) +\frac{M}{\delta^{4}} \mu _{4}\left( \tfrac{1}{n}X_{n}^{x,1-x,-\frac{\min \left\{ x,1-x\right\}}{n-1} }\right) \right) \\
&=\frac{\varepsilon }{n} \sigma ^{2}\left( X_{n}^{x,1-x,-\frac{\min \left\{ x,1-x\right\}}{n-1} }\right) +\frac{M}{n^{3}\delta^4}\mu _{4}\left( X_{n}^{x,1-x,-\frac{\min \left\{ x,1-x\right\}}{n-1} }\right) = \varepsilon O(1) +\frac{1}{\delta ^{4}} O\left( \frac{1}{n}%
\right) ,
\end{align*}%
which can be made arbitrarily small for  $n$ large, concluding the proof.
\end{proof}

\begin{remark} The result given by the previous theorem improves the corresponding result in the case of Bernstein operator (see \cite{Lorentz}, p. 22) by the factor $1-\min\{x,1-x\}$.

$\min \left\{ x,1-x\right\}$

$x \wedge (1-x)$

\end{remark}

\section{Numerical results\label{Section Numerical results}}

We conclude with some numerical and graphical results concerning the
operator $R_{n}$ defined by (\ref{Rational Bernstein operator}). For
comparison, we will use the following well-known Bernstein-type operators:

- the classical Bernstein operator $B_{n}$, given by (\ref{Probabilistic repr of Bernstein polynomial})

- the Lupa\c{s} operator $L_{n}=P_{n}^{\langle 1/n\rangle }$ given by (\ref%
{Lupas operator}), a particular case of Bernstein-Stancu operator $%
P_{n}^{\langle \alpha \rangle }$ $given$ by (\ref{Bernstein-Stancu operator}%
).

- the $q$-Bernstein operator $B_{n,q}$ introduced by Phillips (see \cite%
{Phillips b}, \cite{Phillips})

- the $\left( p,q\right) $-Bernstein operator $S_{n,p,q}$ introduced by
Mursaleen et. al. (see \cite{Mursaleen}).

For the numerical results presented in this section, we used the following
Mathematica program, and similar source codes for the other operators.

\vspace{0.4cm}
{\scriptsize

\texttt{fact[a\_, b\_, k\_]:= If[k = = 0, 1, Product[a + b t, \{t,
0, k - 1\}]];}

\texttt{PolyaProb[a\_, b\_, c\_, n\_, k\_] := Binomial[n, k] fact[a,
c, k] fact[b, c, n - k]/fact[a + b, c, n];}

\texttt{R[x\_, n\_] := Sum[PolyaProb[x, 1 - x, -Min[x, 1 - x]/(n -
1), n, k] f[k/n], \{k, 0, n\}];}
}
\vspace{0.4cm}

In the above, 
the function {\small\texttt{fact[a,b,k]}} computes the rising factorial $a^{\left(
b,k\right) }$ defined by (\ref{rising factorial}), {\small\texttt{%
PolyaProb[a,b,c,n,k]}} computes the probability $p_{n,k}^{a,b,c}$ of P\'{o}lya distribution according to (\ref{Polya urn probabilities}), and {\small\texttt{R[x,n]}} computes the
value of the operator $R_{n}\left( f;x\right) $.

As indicated in \cite{Farouki} (see the footnote on page 385), one
disadvantage of Bernstein operator $B_{n}$ in practical applications is its
slow convergence in case of certain functions. As shown there, in order
to obtain an approximation error less than $10^{-4}$ for the function $f\left(
x\right) =x^{2}$ on $\left[ 0,1\right] $, one needs to consider $n=2500$.
For the same function and the same desired accuracy, numerical computation
show that in case of the operator $R_{n}$ it suffices to consider $n=1250$.
While this number may still be large for certain applications, we
observe that in case of the operator $R_{n}$ the value of $n$ is reduced by
half. To put things differently, for the same value of $n$, the operator $%
R_{n}$ reduces the value of the approximation error of Bernstein operator $%
B_{n}$ by half, while the number of operations needed for evaluating $R_{n}$
and $B_{n}$ are of the same order (as indicated in Section \ref{Section Polya's distribution and probabilistic operators}).

For the graphical comparison of the operator $R_n$ with the operators indicated above, we considered three representative choices of the function $f:[0,1]\rightarrow \mathbb{R}$: a smooth, highly varying function ($f(x)=\sin \left(\frac{9\pi}{2} x\right)$, see Figure~\ref{fig1-2}), a continuous, but only piecewise smooth function ($f(x)=\left \vert 2 \vert x-0.5\vert -0.5\right \vert$, see Figure~\ref{fig3-4}), and a discontinuous function ($f(x)=(x+1)1_{[1/3,1]}(x)$, see Figure~\ref{fig5-6}). For the operators $B_{n,q}$ and $S_{n,p,q}$ we have used the values $p=0.99$ and $q=0.95$ close to $1$, since, as indicated in the corresponding papers (\cite{Phillips b}, \cite{Mursaleen}), they seem to produce better results.

\begin{figure}[ht]
\centering
\subfigure{
\includegraphics[width=0.45\linewidth]{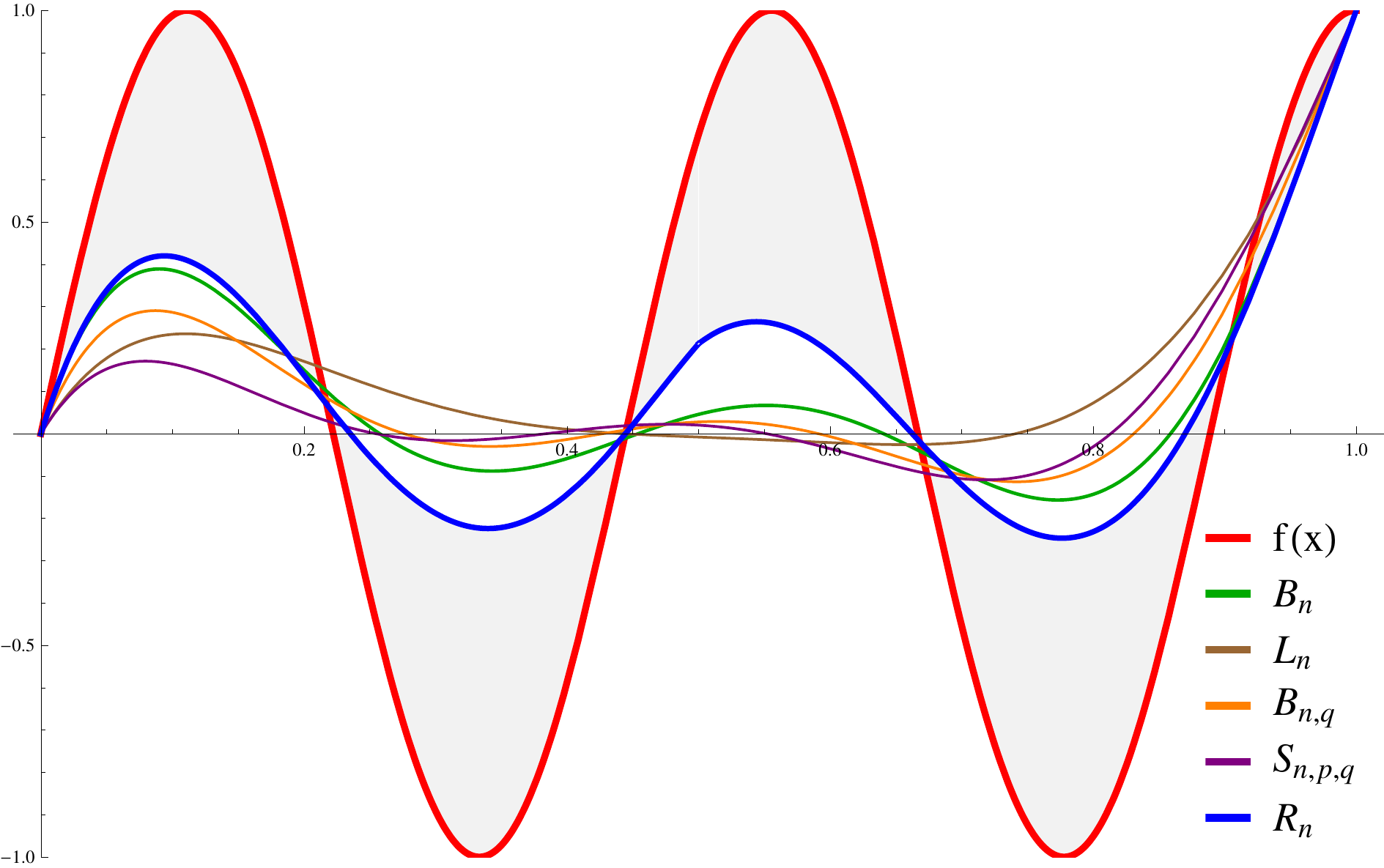}
}
\qquad
\subfigure{
\includegraphics[width=0.45\linewidth]{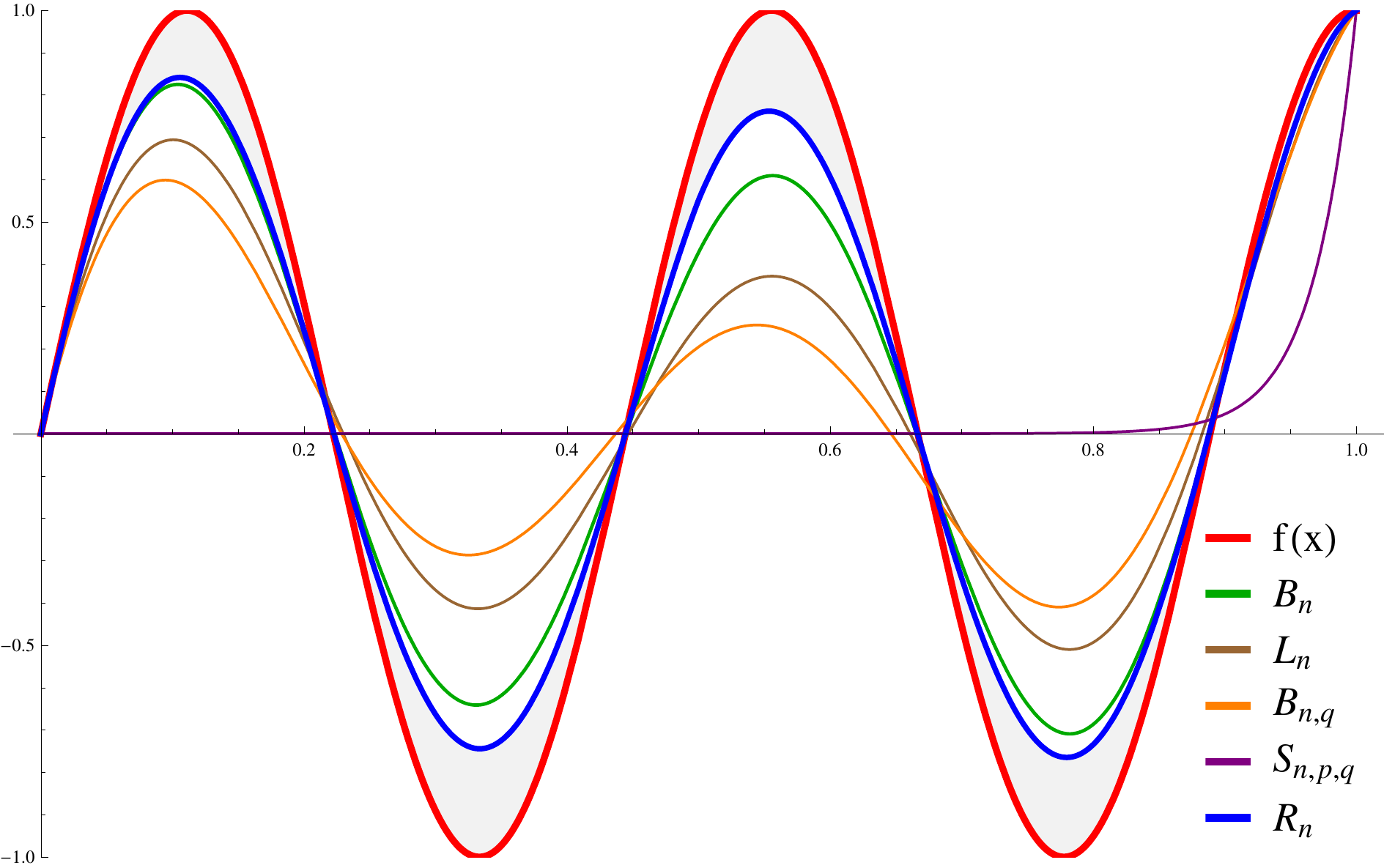}
}
\caption{The approximation of $f(x)=\sin \left(\frac{9\pi}{2} x\right)$, in the case $n=10$ (left) and $n=50$ (right).}
\label{fig1-2}
\end{figure}

\begin{figure}[hb]
\centering
\subfigure{
\includegraphics[width=0.45\linewidth]{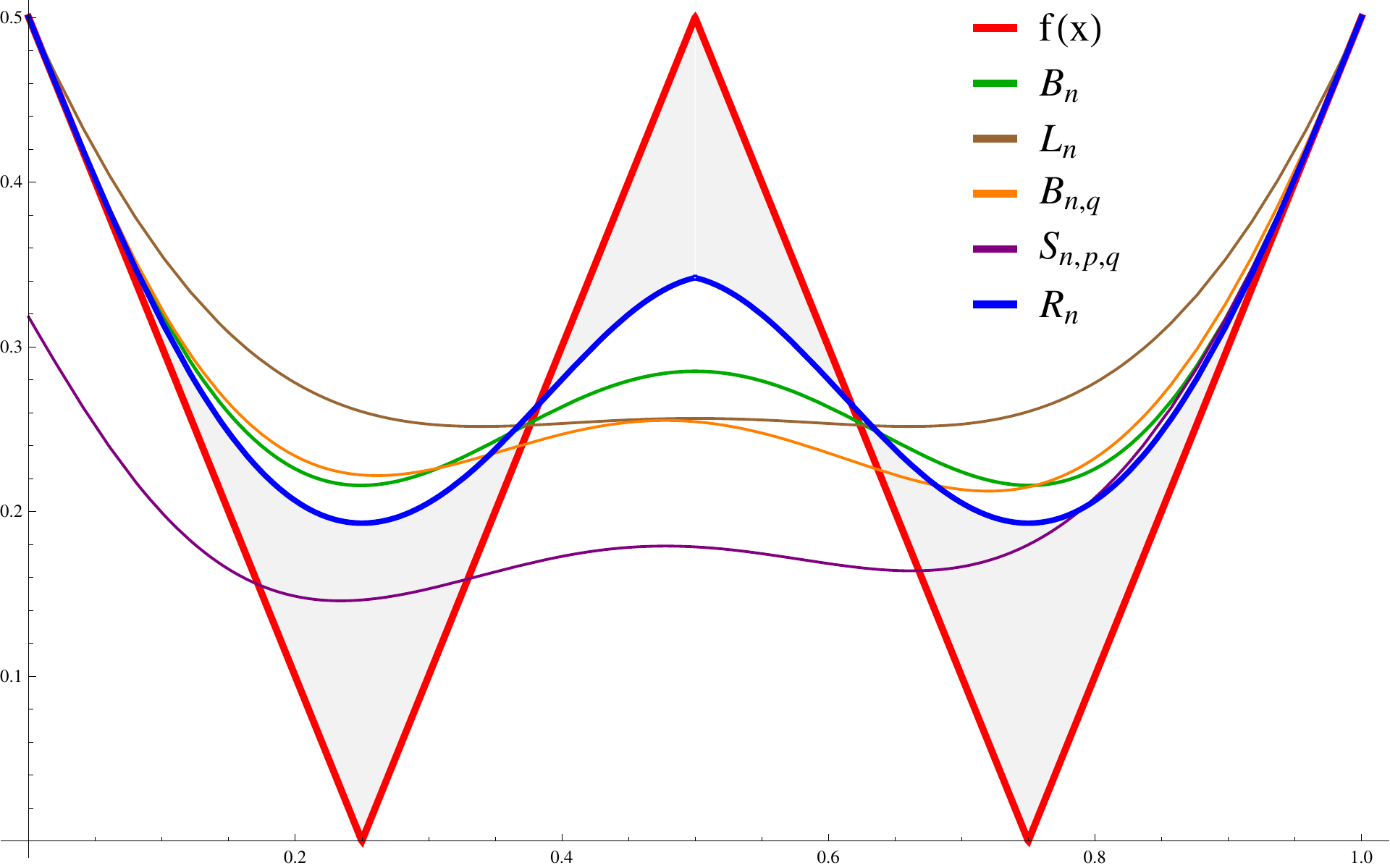}
}
\qquad
\subfigure{
\includegraphics[width=0.45\linewidth]{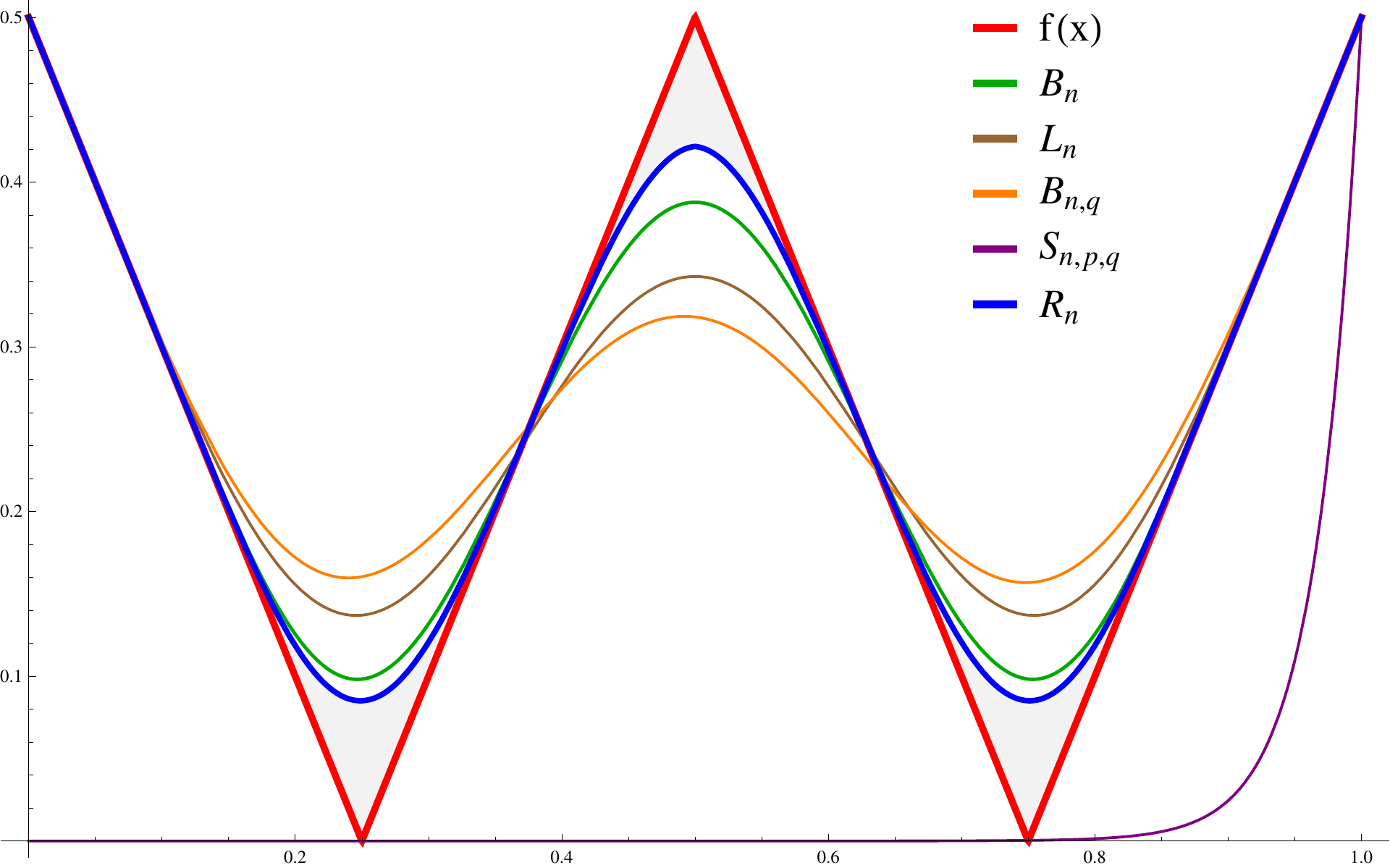}
}
\caption{The approximation of $f(x)=\left\vert  2 \vert x - 0.5\vert - 0.5\right\vert$, in the case $n=10$ (left) and $n=50$ (right).}
\label{fig3-4}
\end{figure}

\begin{figure}[ht]
\centering
\subfigure{
\includegraphics[width=0.45\linewidth]{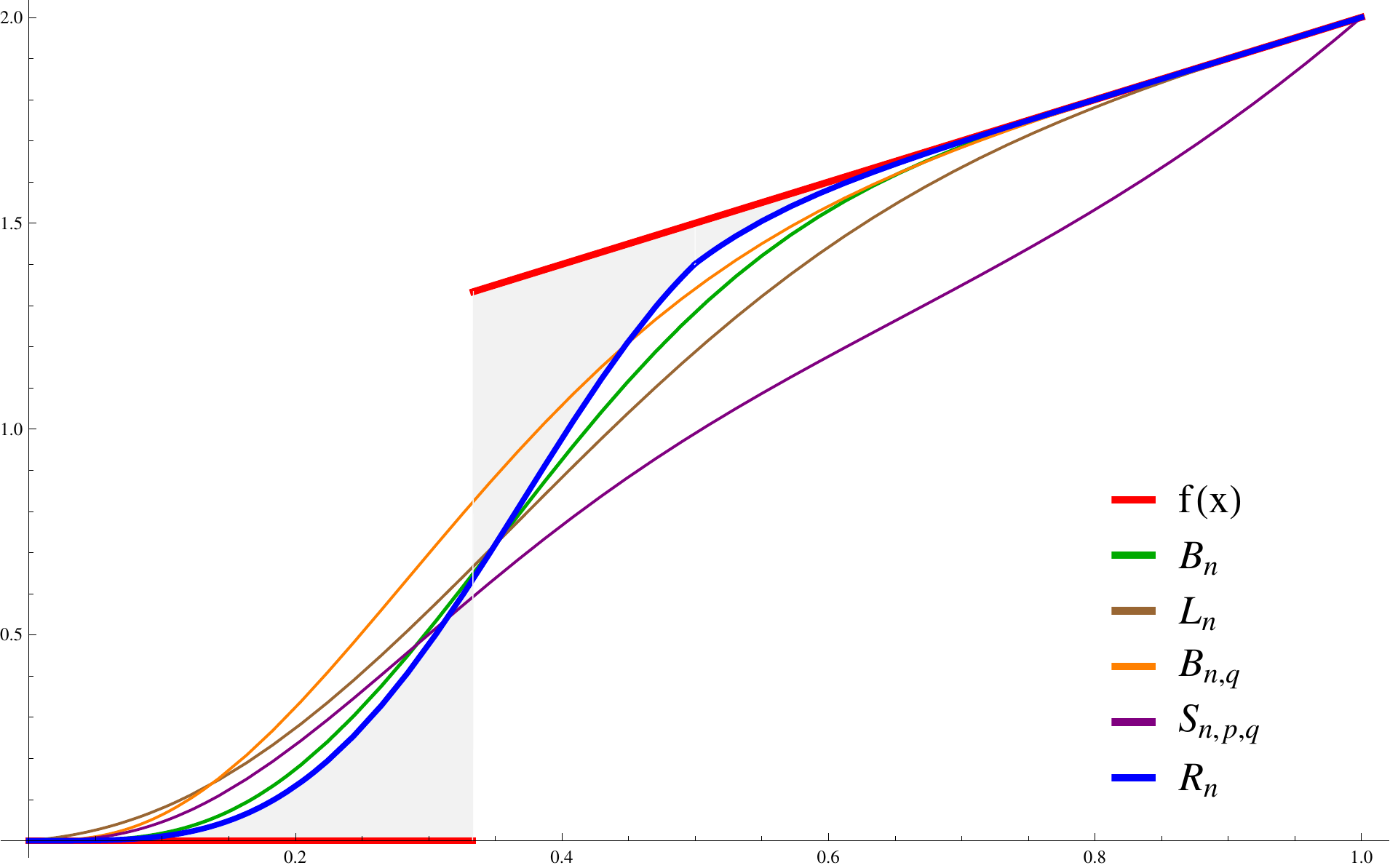}
}
\qquad
\subfigure{
\includegraphics[width=0.45\linewidth]{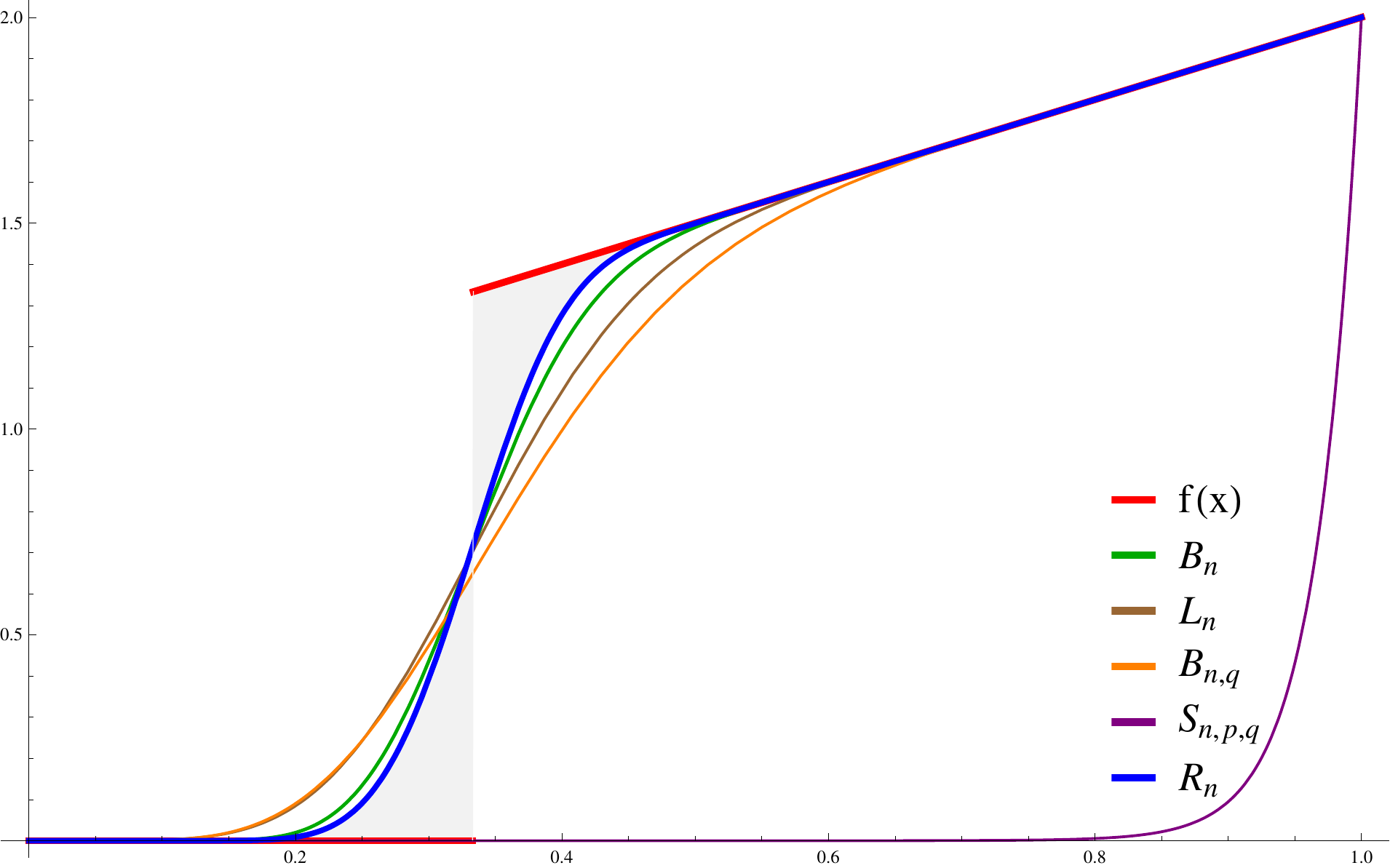}
}
\caption{The approximation of $f(x)=(x+1)1_{[1/3,1]}(x)$ in the case $n=10$ (left) and $n=50$ (right).}
\label{fig5-6}
\end{figure}

The graphical analysis of Figures \ref{fig1-2}, \ref{fig3-4}, and \ref{fig5-6} clearly indicates that the operator $R_n$ provides the best approximation of $f$ in all three cases, followed by the Bernstein operator $B_n$. The ranking of the remaining operators is as follows: for small values of $n$ it appears that $B_{n,q}$ provides a better approximation of $f$, while for larger values of $n$, $L_n$ appears to be better. Although the operator $S_{n,p,q}$ provides a reasonably good approximation of $f$ for small values of $n$, this situation changes for larger values of $n$.


\begin{thebibliography}{99}

\bibitem{Berstein 1912} S. N. Bernstein, D\'{e}monstration du Th\'{e}or\v{c}me de Weierstrass fond\'{e}e sur le calcul des Probabilit\'{e}s, Comm. Soc.Math. Kharkov \textbf{2} (1912), Series XIII, No.1, pp. 1 -- 2.

\bibitem{Bernstein-urn models} S. Bernstein, Sur un probl\`{e}me du sch\'{e}ma des urnes \'{a} composition variable, C. R. (Doklady) Acad. Sci. URSS (N.S.) \textbf{28 }(1940), pp. 5 -- 7.


\bibitem{Eggenberger-Polya} F. Eggenberger, G. P\'{o}lya, \"{U}ber die Statistik verketteter Vorg\"{a}nge, Zeitschrift Angew. Math. Mech. \textbf{3} (1923), pp. 279 -- 289.

\bibitem{Farouki} R. T. Farouki, The Bernstein polynomial basis: a centennial retrospective, Comput. Aided Geom. Design \textbf{29} (2012), No.6, pp. 379 -- 419.

\bibitem{Friedman} B. Friedman, A simple urn model, Comm. Pure Appl. Math. \textbf{2} (1949), pp. 59 -- 70.

\bibitem{Freedman} D.A. Freedman, Bernard Friedman's urn, Ann. Math. Statist. \textbf{36} (1965), pp. 956 -- 970.

\bibitem{Janson} S. Janson, Functional limit theorems for multitype branching processes and generalized P\'{o}lya urns, Stochastic Process. Appl. \textbf{110} (2004), No. 2, pp. 177 -- 245.

\bibitem{Johnson-Kotz} N. L. Johnson, S. Kotz, Urn models and their application. An approach to modern discrete probability theory, Wiley Series in Probability and Mathematical Statistics. John Wiley \& Sons, New York-London-Sydney, 1977.

\bibitem{Agrawal} A. Kajla, N. Ispir, P. N. Agrawal, M. Goyal, $q$--Bernstein--Schurer--Durrmeyer type operators for functions of one and two variables, Appl. Math. Comput. \textbf{275} (2016), pp. 372 -- 385.

\bibitem{Khan} R. A. Khan, Some probabilistic methods in the theory of approximation operators, Acta Math. Acad. Sci. Hungar. \textbf{35} (1980), No. 1 -- 2, pp. 193 -- 203.

\bibitem{Kim} T. Kim, A note on $q$-Bernstein polynomials, Russ. J. Math. Phys. \textbf{18} (2011), No. 1, pp. 73 -- 82.



\bibitem{Lorentz} G. G. Lorentz, Bernstein polynomials (second edition), Chelsea Publishing Co., New York, 1986.

\bibitem{Lupas} L. Lupa\c{s}, A. Lupa\c{s}, Polynomials of binomial type and approximation operators, Studia Univ. Babe\c{s}-Bolyai Mathematica \textbf{32} (1987), No. 4, pp. 61 -- 69.

\bibitem{Lupas-q-Bernstein} A. Lupa\c{s}, A $q$-analogue of the Bernstein operator, Semin. Numer. Stat. Calc. Univ. Cluj-Napoca 9 (1987), pp. 85 -- 92.


\bibitem{Muraru} C. V. Muraru, Note on $q$-Bernstein-Schurer operators, Studia UBB, Mathematica LVI 2 (2011), pp. 1 -- 11.

\bibitem{Mursaleen} M. Mursaleen, K. J. Ansari, A. Khan, Asif, Some approximation results by $(p,q)$-analogue of Bernstein-Stancu operators, Appl. Math. Comput. \textbf{264} (2015), pp. 392 -- 402.

\bibitem{PPT} M. N. Pascu, N. R. Pascu, Approximation properties of a new Bernstein-type operator (to appear).

\bibitem{Popoviciu} T. Popoviciu, Sur l'approximation des functions convexes d'ordre sup\'{e}rieur, Mathematica (Cluj) \textbf{10} (1935), pp. 49 -- 54.

\bibitem{Polya} G. P\'{o}lya, Sur quelques points de la th\'{e}orie des probabilit\'{e}s, Ann. Inst. Poincar\'{e} \textbf{1} (1931), 117 -- 161.

\bibitem{Phillips b} G. M. Phillips, On generalized Bernstein polynomials, Numerical Analysis: A. R. Mitchell 75th Birthday Volume, World Scientific, Singapore, pp. 263 -- 269.

\bibitem{Phillips} G. M. Phillips, Bernstein polynomials based on the $q$-integers,The heritage of P.L.Chebyshev, Ann. Numer. Math. \textbf{4} (1997), pp. 511 -- 518.

\bibitem{Phillips c} G. M. Phillips, A survey of results on the $q$-Bernstein polynomials, IMA J. Numer. Anal. \textbf{30} (2010), No. 1, pp. 277 -- 288.


\bibitem{Schurer} F. Schurer, Linear Positive Operators in Approximation Theory, Math. Inst., Techn. Univ. Delf Report (1962).

\bibitem{Sikkema} P. C. Sikkema, Der Wert einiger Konstanten in der Theorie der Approximation mit Bernstein-Polynomen, Numer. Math. \textbf{3} (1961), pp. 107 -- 116.

\bibitem{Stancu} D. D. Stancu, Approximation of functions by a new class of linear polynomial operators, Rev. Roum. Math. Pures et Appl. \textbf{13} (1968), pp. 1173 -- 1194.

\bibitem{Tonkov} V. O. Tonkov, Addition to the Popoviciu theorem, Translation of Mat. Zametki \textbf{94} (2013), No. 3, pp. 416 -- 425. Math. Notes \textbf{94} (2013), No. 3 -- 4, pp. 392 -- 399.
\end{thebibliography}
\end{document}